\documentclass[11pt,a4paper]{amsart}

\usepackage{epsfig}
\usepackage{graphicx}
\usepackage[usenames,dvipsnames]{xcolor}
\usepackage[colorlinks=true,linkcolor=Blue,citecolor=Green,urlcolor=Black]{hyperref}
\usepackage{amsmath,amsthm,amsfonts,amssymb}
\usepackage{commath}	
\usepackage{nicefrac}
\usepackage[T1]{fontenc}
\usepackage[utf8]{inputenc}
\usepackage[textwidth=2.5cm,textsize=small]{todonotes}
\usepackage[capitalise,nameinlink,noabbrev]{cleveref}   
\crefname{enumi}{}{}
\usepackage[alwaysadjust]{paralist}

\newtheorem{theorem}{Theorem}[section]
\newtheorem{lemma}[theorem]{Lemma}
\newtheorem{proposition}[theorem]{Proposition}

\newtheorem{corollary}[theorem]{Corollary}
\newtheorem{conjecture}[theorem]{Conjecture}

\newtheorem{remark}[theorem]{Remark}
\theoremstyle{definition}
\newtheorem{definition}[theorem]{Definition}

\DeclareMathOperator{\conv}{conv}

\DeclareMathOperator{\bd}{bd}
\DeclareMathOperator{\diag}{diag}

\def\K{\mathcal{K}}
\def\D{\mathrm{D}}

\def\O{\mathcal{O}}
\def\calS{\mathcal{S}}
\def\R{\mathbb{R}}

\def\B{\mathbb{B}}
\def\bS{\mathbb{S}}
\def\N{\mathbb{N}}

\def\V{\mathrm{vol}}

\def\w{\mathrm{w}}
\def\I{\mathrm{I}}

\def\eps{\varepsilon}
\def\wD{\mathrm{W}}
\def\e{\mathrm{e}}

\DeclareMathOperator{\vol}{vol}

\DeclareMathOperator{\lin}{lin}

\DeclareMathOperator{\GL}{GL}
\DeclareMathOperator{\NO}{NO}

\DeclareMathOperator{\iq}{idq}
\DeclareMathOperator{\iwq}{iwq}
\DeclareMathOperator{\DR}{DR}

\begin{document}

\title[On the reverse isodiametric inequality]{On the reverse isodiametric problem and Dvoretzky-Rogers-type volume bounds}
\author{Bernardo Gonz\'{a}lez Merino}
\address{Departamento de An\'alisis Matem\'atico, Facultad de Matem\'aticas, Universidad de Sevilla, Apdo. 1160, 41080-Sevilla, Spain}
\email{bgonzalez4@us.es}

\author{Matthias Schymura}
\address{BTU Cottbus-Senftenberg\\
  Platz der Deutschen Einheit 1\\
  03046 Cottbus\\
  Germany}
\email{schymura@b-tu.de}

\thanks{
The research of the first author is a result of the activity developed within the framework of the Programme in Support of Excellence Groups of the Regi\'{o}n de Murcia, Spain, by Fundaci\'{o}n S\'{e}neca, Science and Technology Agency of the Regi\'{o}n de Murcia. He was partially supported by Fundaci\'{o}n S\'{e}neca project 19901/GERM/15, MINECO project MTM2015-63699-P, and MICINN Project PGC2018-094215-B-I00, Spain.
The second author was partially supported by the Freie Universit\"at Berlin within the Excellence Initiative of the German Research Foundation.}

\date{\today}



\begin{abstract}
The isodiametric inequality states that the Euclidean ball maximizes the volume among all convex bodies of a given diameter.
We are motivated by a conjecture of Makai Jr.~on the reverse question: Every convex body has a linear image whose isodiametric quotient is at least as large as that of a regular simplex.
We relate this reverse isodiametric problem to minimal volume enclosing ellipsoids and to the Dvoretzky-Rogers-type problem of finding large volume simplices in any decomposition of the identity matrix.

As a result, we solve the reverse isodiametric problem for $o$-symmetric convex bodies and obtain a strong asymptotic bound in the general case.
Using the Cauchy-Binet formula for minors of a product of matrices, we obtain Dvoretzky-Rogers-type volume bounds which are of independent interest.
\end{abstract}

\maketitle


\section{Introduction}

Let $\K^n$ be the family of \emph{convex bodies} in $\R^n$, that is, full-dimensional convex compact sets.
If $K=-K$, we say that $K$ is \emph{$o$-symmetric}, and we denote by $\K^n_o$ the family of all such convex bodies.
Further, we denote by $\D(K)=\max\{\|x-y\| : x,y \in K\}$ the diameter and by $\V_n(K)$ the volume of $K\in\K^n$.
If the dimension is clear from the context, we just write~$\V(K)$.

A classic result in convexity is the \emph{isodiametric inequality} due to Bieberbach~\cite{bieberbach1915ueber} (cf.~\cite[Sect.~8.2]{gruber2007convex}) which states that the Euclidean unit ball~$\B_2^n$ has the maximum volume among all convex bodies of the same diameter.
In other words, $\B_2^n$ maximizes the \emph{isodiametric quotient}, more precisely
\begin{align}
\iq(K) := \frac{\V(K)}{\D(K)^n} \leq \frac{\V(\B_2^n)}{\D(\B_2^n)^n},\quad\textrm{for all }K\in\K^n.\label{eqn_isodiametric_inequality}
\end{align}
Boxes of arbitrarily large diameter and constant volume show that $\iq(K)$ is in general not bounded from below by any constant $c>0$ that only depends on the dimension~$n$.
In the spirit of the reverse isoperimetric inequality by Ball~\cite{ball1991volume}, it is natural to ask whether there is a suitable linear transformation~$A$ such that the linear image~$AK$ of~$K$ has an isodiametric quotient that \emph{can} be bounded away from zero.

Makai Jr.~\cite{makai1978on} posed the conjecture that there is always a linear image whose isodiametric quotient is at least as large as that of a regular simplex.
He was motivated by an application to the minimal density of non-separable lattice arrangements of convex bodies, and based his conjecture on the solution of the
reverse isodiametric problem in the plane, which was found by Behrend~\cite{behrend1937ueber} already in 1937.

\begin{conjecture}[{Makai Jr.~\cite{makai1978on}}]\label{conj_rev_idi}
For every $K\in\K^n$ there is a linear transformation $A \in \GL_n(\R)$ such that
\begin{align}
\iq(AK) \geq \frac{\sqrt{n+1}}{n!\,2^{n/2}},\label{eqn_conj_gen_case}
\end{align}
with equality sign if and only if $AK$ is a regular simplex.

If we assume that $K\in\K^n_o$, then an $A \in \GL_n(\R)$ exists, such that
\begin{align}
\iq(AK) \geq \frac{1}{n!},\label{eqn_conj_sym_case}
\end{align}
with equality sign if and only if $AK$ is a regular crosspolytope.
\end{conjecture}

To be more precise, a \emph{regular simplex} is a simplex all of whose edges have the same length.
A \emph{regular crosspolytope} in~$\R^n$ is the convex hull of $\pm u_1,\ldots,\pm u_n$, where $\{u_1,\ldots,u_n\}$ is an orthonormal basis of~$\R^n$.
While \cref{conj_rev_idi} is open for any dimension $n\geq3$, Makai Jr.~\cite[Lem.~2]{makai1978on} proved that there is always some $A \in \GL_n(\R)$ such that
\begin{align}
\iq(AK) &\geq \frac{\V(\conv\{\B^n_2,\pm(\sqrt{n},0,\ldots,0)^\intercal\})}{\binom{2n}{n}n^{n/2}}\approx\frac{\sqrt{n+1}}{n!\,8^{n/2}}.\label{eqn_Makai_estimate}
\end{align}
In this work, we relate the reverse isodiametric problem to minimal volume enclosing ellipsoids and to the Dvoretzky-Rogers-type problem of finding large volume simplices in any decomposition of the identity matrix.

The central definition for our investigations is the following:

\begin{definition}\label{def_Behrend_position}
A convex body $K\in\K^n$ is in \emph{isodiametric position} (or \emph{Behrend position}), if
\[
\iq(K) = \max_{A\in\GL_n(\R)} \iq(AK).
\]
\end{definition}

This definition is justified by standard arguments in convexity that show that the supremum of the isodiametric quotient of linear images of a fixed $K \in \K^n$ is always attained (see \cref{lem_attainment_Behrendposition}).
We prove in \cref{thm_uniqueness_Behrend_pos} that the Behrend position is unique up to rotations, scalings and translations.

In \cref{sect_behrend_loewner}, we make the crucial observation that a convex body~$K$ is in Behrend position if and only if its normalized difference body $(K-K)/\D(K)$ is in L\"owner position (see \cref{thm_equival_isod_lown}), which means that its volume minimal enclosing ellipsoid is the Euclidean unit ball.
This relationship allows to use a result of Barthe~\cite{barthe1998an} on the minimal volume of an $o$-symmetric convex body in L\"owner position, solving \cref{conj_rev_idi} in the $o$-symmetric case.

\begin{proposition}\label{thm_o-symm_rev_isod_ineq}
Let $K\in\K^n_o$ be in Behrend position. Then
\[
\iq(K)\geq\frac1{n!}.
\]
Equality holds if and only if $K$ is a regular crosspolytope.
\end{proposition}

Behrend observed that the directions of the line segments attaining the diameter of a planar convex body in isodiametric position correspond to a well-distributed point configuration on the unit circle.
In \cref{sect_distribution_lemma}, we show how his ideas can be extended to arbitrary dimension and use this information to significantly improve the asymptotic bound~\eqref{eqn_Makai_estimate} as follows.
The notation $f(n) \sim g(n)$ means that $\lim_{n \to \infty} f(n) / g(n) = 1$.

\begin{theorem}\label{thm_gen_rev_isod_ineq}
Let $K\in\K^n$ be in Behrend position. Then
\[
\iq(K)\geq\frac{1}{\sqrt{n!}\,n^{n/2}} \sim \left(\frac{2\pi}{n}\right)^{1/4} \frac{\sqrt{n+1}}{n!\,e^{n/2}}.
\]
\end{theorem}


An extremely useful result in Convex Geometry is the characterization of the L\"owner position in terms of the existence of a decomposition of the identity matrix as a non-negative linear combination of rank-one matrices.
Several authors have contributed to this characterization, starting with Behrend~\cite{behrend1937ueber} in the planar case, John \cite{john2014extremum} (whose original work goes back to 1948) in general dimension, and Ball \cite{ball1992ellipsoids}, who clarified the characterization of the ellipsoid by specific touching conditions (see also \cite{henk2012loewnerjohn,tomczak1989banach} for other historical references).
\cref{thm_equival_isod_lown} shows that, for convex bodies in Behrend position, such a decomposition is induced by the directions of line segments attaining the diameter.
The proof of \cref{thm_gen_rev_isod_ineq} in \cref{sect_distribution_lemma} uses crucially that we can find such diametrical directions which span a simplex of large volume.

These observations motivate our studies in \cref{sect_DR_bounds}, where we are interested in the following problem:
Given a decomposition of the $n \times n$ identity matrix into a sum of $m$ rank-one matrices of the form $uu^\intercal$, find $j$ of the decomposing vectors $u$ that together with the origin span a $j$-dimensional simplex of large volume.
The famous Dvoretzky-Rogers lemma from~\cite{dvoretzky1950Absolute} gives an estimate for the case $j=n$ which is however not sensitive to the parameter~$m$.
Writing $\DR(m,n,j)$ for the largest possible volume that can always be guaranteed (see \cref{Prob_dvorz_rog} for a precise definition), we use the Cauchy-Binet formula for the minors of a product of matrices and prove

\begin{theorem}\label{thm_improved_Dvor_rog_ineq}
Let $1 \leq j \leq n \leq m \leq \binom{n+1}{2}$. Then
\begin{align}
\DR(m,n,j)^2 \geq  \frac{\binom{n}{j}\left(\frac{m}{n}\right)^j}{\binom{m}{j}(j!)^2}.\label{eqn_intro_DR_bounds}
\end{align}
Moreover, the inequality is best possible for the triples
\begin{itemize}
 \item $(n,n,j)$, for $1\leq j\leq n$,
 \item $(n+1,n,j)$, for $1\leq j\leq n$, and
 \item $(\binom{n+1}{2},n,2)$, for $n\in\{2,3,7,23\}$, but not for any other $n \leq 118$.
\end{itemize}
\end{theorem}

The sharpness question for the triple $(\binom{n+1}{2},n,2)$ is related to the existence of a set of many equiangular lines in dimension~$n$.
In the particular case $(m,n,j)=(m,n,n)$, the bound~\eqref{eqn_intro_DR_bounds} was proven in \cite{pelczynski1991parallelepipeds} by similar arguments as ours, and recently was independently obtained in~\cite{fodornaszoditamas2018on} by probabilistic methods.
As a corollary to these Dvoretzky-Rogers-type bounds, we get a second proof of the asymptotic estimate in \cref{thm_gen_rev_isod_ineq} (see~\cref{cor:asymptoticBound_viaCB}).

We complement the discussion on the reverse isodiametric problem by studying its dual counterpart in \cref{sect_isominwidth_position}.
Replacing the diameter of $K$ by the minimum width $\w(K)$, we define the \emph{isominwidth quotient} by
\[
\iwq(K):=\frac{\V(K)}{\w(K)^n}.
\]
We then consider the reverse isominwidth problem, which asks for an upper bound on the minimum isominwidth quotient of a linear image of any given~$K \in \K^n$.
The strong duality between the diameter and the minimum width implies characterization results regarding $\iwq(K)$ that are analogous to those in \cref{sect_behrend_loewner}.

Finding good lower bounds on the quotient $\iwq(K)$ is an intricate and longstanding problem, most commonly known as P\'al's problem.
However, based on the experiences we made concerning the Behrend position, we are able to give a complete solution to the reverse isominwidth problem, which generalizes the respective statement of Behrend~\cite[p.~716, (II$_2$)]{behrend1937ueber} about the planar case.

\begin{theorem}\label{thm_isominwidth_reverse_ineq}
Let $K\in\K^n$. Then there exists $A \in \GL_n(\R)$ such that
\[
\iwq(AK)\leq 1.
\]
Moreover, equality holds if and only if $AK$ is a cube.
\end{theorem} 

\section{Convex bodies in Behrend position and the \texorpdfstring{$o$}{o}-symmetric reverse isodiametric inequality}\label{sect_behrend_loewner}

In this section, we establish a close relationship between the Behrend position and the well-known L\"owner position of a convex body.
As a result we obtain the solution to the reverse isodiametric problem for $o$-symmetric convex bodies.

Let us first justify the definition of the Behrend position by showing that the supremal isodiametric quotient among the linear images of a fixed convex body is always attained. We prove a standard compactness argument, analogous to the ones that had been proven
before for other special positions of convex sets (cf.~\cite{artstein2015asymptotic}).
We refer the reader to the textbook by Gruber~\cite[Ch.~6]{gruber2007convex} for information on the set of convex bodies as a metric space.

\begin{lemma}\label{lem_attainment_Behrendposition}
For every $K \in \K^n$, there exists an $A \in \GL_n(\R)$ such that
\[
\iq(AK) = \sup_{B \in \GL_n(\R)}\iq(BK).
\]
\end{lemma}
\begin{proof}
First observe that by the scaling- and translation-invariance of the isodiametric quotient it suffices to consider $K \in \K^n$ containing the origin in their interior and linear maps $A$ that are volume-preserving.
Therefore,
\begin{align}
\sup_{A\in\GL_n(\R)}\frac{\V(AK)}{\D(AK)^n} = \frac{\V(K)}{\inf_{A \in \calS_K} \D(AK)^n},\label{eqn_attainment_1}
\end{align}
where
\[
\calS_K = \bigg\{A \in \GL_n(\R) : \det(A)=1 \text{ and } 2\left(\frac{\V(K)}{\kappa_n}\right)^{1/n} \leq \D(AK) \leq \D(K)\bigg\}.
\]
Note, that the lower bound on the diameter of $AK$ follows from~\eqref{eqn_isodiametric_inequality}.

Now, take a sequence of convex bodies $AK$, $A \in \calS_K$, whose diameters converge to the infimum in~\eqref{eqn_attainment_1}.
As we have fixed the origin to be contained in $K$ and by the definition of $\calS_K$, this sequence is bounded in the sense that all of its members are contained in a ball of diameter $\D(K)$.
In view of Blaschke's selection theorem (cf.~\cite[Thm.~6.3]{gruber2007convex}), there exists a convergent subsequence with limit $\bar K = \bar A K$, for $A \to \bar A \in \calS_K$.
By the continuity of the diameter function with respect to the Hausdorff distance, we get that $\D(\bar A K)^n = \inf_{A \in \calS_K} \D(AK)^n$, finishing the proof.
\end{proof}

\subsection{Behrend position versus L\"owner position}

In the sequel, we say that two points~$x$ and~$y$ in a convex body $K\in\mathcal K^n$ determine a \emph{diametrical segment} of~$K$ if $\D(K)=\|x-y\|$, and in this case we say that $(x-y)/\|x-y\|$ is a \emph{diametrical direction}.
We denote by
\begin{align*}
\D_K &= \left\{ u \in \bS^{n-1} : u\text{ is a diametrical direction of }K \right\}\\
& = \left\{ u \in \bS^{n-1} : \exists\,x\in K\text{ such that }x+\D(K)[0,u]\subseteq K \right\}
\end{align*}
the set of all diametrical directions of~$K$.
Note that $\D_K$ is $o$-symmetric.

A convex body $K\in\K^n$ is in \emph{L\"owner position} if $\B^n_2$ is
a minimum volume ellipsoid containing~$K$.
For background information, references, and a discussion of the history regarding the L\"owner (and John) position we refer the reader to the survey article by Henk~\cite{henk2012loewnerjohn}.

It is well-known that for every $K\in\K^n$, there exists an $A \in \GL_n(\R)$ and a translation $t \in \R^n$ such that $AK + t$ is in L\"owner position, and that the minimal volume ellipsoid containing $K$, called the \emph{L\"owner ellipsoid} of~$K$, is unique (cf.~\cite{artstein2015asymptotic,henk2012loewnerjohn}).
Moreover, the L\"owner position of a convex body is characterized by the existence of contact points that decompose the $n\times n$ identity matrix~$\I_n$.
More precisely,

\begin{theorem}[cf.~{\cite[Ch.~11]{gruber2007convex}}]\label{thm_lowner_position}
Let $K\in\K^n$ be such that $K\subseteq\B^n_2$.
The following are equivalent:
\begin{enumerate}[(i)]
 \item $K$ is in L\"owner position.
 \item There exists an $m \geq n$, contact points $u_1,\ldots,u_m \in \bd(K)\cap\bS^{n-1}$, and scalars $\lambda_1,\ldots,\lambda_m\geq 0$ such that
\[
\I_n=\sum_{i=1}^m\lambda_i u_iu_i^\intercal \qquad \text{and} \qquad \sum_{i=1}^m \lambda_i u_i = 0.
\]
\end{enumerate}
Moreover, if $K$ is $o$-symmetric, then the condition $\sum_{i=1}^m \lambda_i u_i = 0$ can be dropped, and one can choose $m \leq \binom{n+1}{2}$.
\end{theorem}

\noindent We are now set in order to state the main result of this section.

\begin{theorem}\label{thm_equival_isod_lown}
Let $K\in\K^n$.
The following are equivalent:
\begin{enumerate}[(i)]
\item $K$ is in Behrend position.
\item $K-K$ is in Behrend position.
\item $(K-K)/\D(K)$ is in L\"owner position.
\item $\conv(\D_K)$ is in L\"owner position.
\item There exists an $m\in\{n,\dots,\binom{n+1}{2}\}$, scalars $\lambda_1,\ldots,\lambda_m\geq 0$, and diametrical directions $u_1,\ldots,u_m \in \D_K$, such that
\[
\I_n=\sum_{i=1}^m\lambda_i u_iu_i^\intercal.
\]
\end{enumerate}
\end{theorem}

The proof of \cref{thm_equival_isod_lown} rests on two key lemmas for which we introduce some notation.
We write $\NO_1(n)$ for the set of non-orthogonal matrices $M \in \GL_n(\R) \setminus \O(n)$ with $\abs{\det(M)}=1$.
Geometrically, $\NO_1(n)$ contains all volume-preserving linear maps that do not keep the unit ball $\B_2^n$ invariant.
The following lemma is certainly folklore in the literature around volume-minimizing enclosing ellipsoids.
For completeness we provide the proof.

\begin{lemma}\label{lem_low_det_1}
Let $K\in\K_o^n$ be such that $K\subseteq\B^n_2$.
The following are equivalent:
\begin{enumerate}[(i)]
\item $K$ is in L\"owner position.
\item For every $M \in \NO_1(n)$, we have $K\nsubseteq M(\B^n_2)$.
\end{enumerate}
\end{lemma}

\begin{proof}
In order to show (i) $\Rightarrow$ (ii), we use the fact that $\B^n_2$ is the \emph{unique} ellipsoid of volume $\V(\B^n_2)$ containing $K$.
Therefore, if $M \in \NO_1(n)$, then $M(\B^n_2)\neq\B^n_2$ and $\V(M(\B^n_2))=\V(\B^n_2)$, and hence we get $K\nsubseteq M(\B^n_2)$.

Now we show (ii) $\Rightarrow$ (i).
Let $M\in\GL_n(\R)$ be with $\abs{\det(M)}\leq 1$.
If $\abs{\det(M)}=1$, then either $M\in\O(n)$ (and then $M(\B^n_2)=\B^n_2$), or $M\notin\O(n)$, which then by (ii) implies $K\nsubseteq M(\B^n_2)$.
If, on the contrary, $\abs{\det(M)}<1$, we write $N:=\abs{\det(M)}^{-1/n}M$ and observe that $\abs{\det(N)}=1$.
Again, if $N\notin\O(n)$, then $K\nsubseteq N(\B^n_2) = \abs{\det(M)}^{-1/n}M(\B^n_2)$.
Since $\abs{\det(M)}<1$, we also have that $K\nsubseteq M(\B^n_2)$.
Finally, we suppose that $N\in\O(n)$. In order to verify that $K\nsubseteq M(\B^n_2)$, we make use of the fact that there exists a touching point $u\in K\cap \bS^{n-1}$ (we verify it at the end of the proof).
Indeed, under this assumption, in view of $\bS^{n-1}=N(\bS^{n-1})$, we have
$u\in K\cap N(\bS^{n-1}) = K \cap \abs{\det(M)}^{-1/n} M (\bS^{n-1})$, therefore $u \notin M(\B^n_2)$, and thus $K \nsubseteq M(\B^n_2)$, concluding the proof of (i).

As promised, we show that (ii) implies that $K\cap\bS^{n-1}\neq\emptyset$.
If, on the contrary, $K\subseteq\rho\B^n_2$, for some $\rho\in(0,1)$, we can consider the matrix $M_{\varepsilon}:=\diag(\varepsilon,\varepsilon^{-1},1,\dots,1)\in\GL_n(\R)\setminus\O(n)$, for $\varepsilon\in(0,1)$, which of course satisfies $\det(M_{\varepsilon})=1$.
Since $\lim_{\varepsilon\rightarrow 1}M_{\varepsilon}=\I_n$, there exists $\varepsilon_0$ close enough to $1$ such that $\rho\B^n_2\subseteq M_{\varepsilon_0}(\B^n_2)$, therefore implying that $K\subseteq \rho\B^n_2\subseteq M_{\varepsilon_0}(\B^n_2)$, and thus contradicting (ii).
\end{proof}

The second lemma has been shown by Behrend~\cite[Satz 7u \& 11u]{behrend1937ueber} in the case of the plane $n=2$.
Using \cref{lem_low_det_1}, Behrend's result extends to any dimension.

\begin{lemma}\label{lem_lown_isodiam}
Let $K\in\K_o^n$ be such that $\D(K)=2$.
The following are equivalent:
\begin{enumerate}[(i)]
\item $K$ is in Behrend position.
\item $K$ is in L\"owner position.
\end{enumerate}
\end{lemma}
\begin{proof}
Since $\D(K) = 2$, we have $K \subseteq \B^n_2$.
Therefore, using \cref{lem_low_det_1},
\begin{align*}
& \phantom{\Leftrightarrow\,\,}\, K\text{ is in L\"owner position }\\
& \Leftrightarrow\, \text{for all } M \in \NO_1(n)\,\,\exists\, u\in K\text{ such that }\|M(u)\|\geq 1\\
& \Leftrightarrow\, \text{for all } M \in \NO_1(n): \D(M(K)) \geq 2 = \D(K)\\
& \Leftrightarrow\, K\text{ is in Behrend position.}\qedhere
\end{align*}
\end{proof}

\begin{proof}[Proof of \cref{thm_equival_isod_lown}]
(i) $\Leftrightarrow$ (ii). Since $M(K-K)=M(K)-M(K)$ for every $M \in \GL_n(\R)$, and $\D(K-K)=2\D(K)$, we can conclude that
\begin{align*}
& \phantom{\Leftrightarrow\,\,}\, K\text{ is in Behrend position}\\
& \Leftrightarrow\, \text{for all } M \in \NO_1(n): \D(M(K)) \geq \D(K)\\
& \Leftrightarrow\, \text{for all } M \in \NO_1(n): \D(M(K-K)) \geq \D(K-K)\\
& \Leftrightarrow\, K-K\text{ is in Behrend position.}
\end{align*}

(ii) $\Leftrightarrow$ (iii). This follows from \cref{lem_lown_isodiam}.

(iii) $\Leftrightarrow$ (iv). Observe that $u \in ((K-K)/\D(K)) \cap \bS^{n-1}$ if and only if $u\in \D_K$.
This means that the contact points $u_i$, $i\in[m]$, in \cref{thm_lowner_position}, belong to both $(K-K)/\D(K)$ and $\D_K$, which shows the claimed equivalence.

(iv) $\Leftrightarrow$ (v). Apply \cref{thm_lowner_position} to the body $\conv(D_K)$.
\end{proof}

In contrast to the Behrend position, it is in general not true that $K \in \K^n$ is in L\"owner position if and only if $(K-K)/\D(K)$ is.
The following proposition provides examples showing that in fact neither implication holds in general.

\begin{proposition}\
\begin{enumerate}[(i)]
 \item For $\frac{\sqrt{3}}{2} < r \leq 1$, the ``sailing boat''
\[
K_r= \conv\bigg\{\binom{0}{1},\binom{\pm\frac{\sqrt{3}}{2}}{-\frac12},\binom{\pm \sqrt{1-r^2}}{-r}\bigg\}
\]
is in L\"owner position, but $(K_r-K_r)/\D(K_r)$ is not.

 \item For $r \in [0,1)$, let $T_r = \conv\left\{(0,1)^\intercal,(\pm \sqrt{1-r^2},-r)^\intercal\right\}$ be a triangle, and let $Q_2 = \conv\{(\pm 1/\sqrt{2}, \pm 1/\sqrt{2})^\intercal\}$ be the axes parallel square inscribed into the unit circle.
Then, for $\varepsilon > 0$ small enough, the septagon
\[
K_\varepsilon = \conv\left\{(1-\varepsilon)Q_2, T_{1/2-\varepsilon}\right\}
\]
is not in L\"owner position, but $(K_\varepsilon-K_\varepsilon)/\D(K_\varepsilon)$ is.
\end{enumerate}
\end{proposition}

\begin{proof}
(i): Using \cref{thm_lowner_position}, one checks that the equilateral triangle~$T_2$ with vertices $(0,1)^\intercal$, $(\sqrt{3}/2,-1/2)^\intercal$, and $(-\sqrt{3}/2,-1/2)^\intercal$ is in L\"owner position (cf.~\cite[\S 7]{behrend1937ueber}).
Since $T_2 \subseteq K_r \subseteq \B^2_2$, the pentagon $K_r$ is in L\"owner position as well.

Now, $\D(K_r) = \|(0,1)^\intercal - (\pm\sqrt{1-r^2},-r)^\intercal \| = \sqrt{2(r+1)}$, and moreover the segments $[(0,1)^\intercal,(\pm\sqrt{1-r^2},-r)^\intercal]$ are the only diametrical segments of~$K_r$.
Hence, for $\sqrt{3}/2 < r \leq 1$,
\[
\left(\frac{1}{\D(K_r)}(K_r-K_r)\right) \cap \bS^1 = \left\{ \frac{1}{\sqrt{2(r+1)}}\left(\pm\sqrt{1-r^2},\pm(r+1)\right)^\intercal \right\},
\]
with the two $\pm$ signs being independent.
Therefore, the arc with midpoint~$e_1$ and angle $2\pi/3$ of the circle $\bS^1$ contains no diametrical direction of~$K_r$.
Thus, the diameter condition (\cref{lem_behrend}, cited from~\cite{behrend1937ueber}) implies that $(K_r-K_r)/\D(K_r)$ is not in Behrend position.
In view of \cref{thm_equival_isod_lown} this shows that $(K_r-K_r)/\D(K_r)$ is not in L\"owner position.

(ii): First of all, the unit circle $\B^2_2$ is the smallest circle containing $T_{1/2-\varepsilon}$, and hence it is the smallest circle containing $K_\varepsilon$.
The contact points $\bd(K_\varepsilon) \cap \bS^1$ are exactly the vertices of the triangle $T_{1/2-\varepsilon}$.
The reader quickly convinces herself that for $\varepsilon > 0$ these contact points do not provide a decomposition of the identity $\I_2$ according to \cref{thm_lowner_position}, so that $K_\varepsilon$ is not in L\"owner position (cf.~\cite[\S7]{behrend1937ueber}).

On the other hand, for $\varepsilon$ small enough, the diameter of $K_\varepsilon$ is attained by the opposite pairs of vertices of $(1-\varepsilon)Q_2$.
Therefore, $Q_2 \subseteq (K_\varepsilon-K_\varepsilon)/\D(K_\varepsilon) \subseteq \B^2_2$, and since $Q_2$ is in L\"owner position (cf.~\cite[Satz 11u]{behrend1937ueber}), the difference body $(K_\varepsilon-K_\varepsilon)/\D(K_\varepsilon)$ is as well.
\end{proof}

Based on \cref{thm_equival_isod_lown}, we can now give a succinct characterization of uniqueness of the Behrend position of a convex body.

\begin{proposition}\label{thm_uniqueness_Behrend_pos}
The Behrend position of a convex body is unique up to orthogonal transformations, scalings, and translations.
\end{proposition}
\begin{proof}
The isodiametric quotient of a convex body is clearly invariant under orthogonal transformations, scalings, and translations.
Hence, the property of a convex body to be in Behrend position is invariant under these transformations as well.

In order to show that this is an exhaustive list of such transformations, it suffices to consider $o$-symmetric convex bodies.
In fact in view of \cref{thm_equival_isod_lown}, $K$ is in Behrend position if and only if its difference set $K-K$ is.
Therefore, let $K \in \K_o^n$ be in Behrend position and furthermore let $K \subseteq \B^n_2$ and $\D(K)=2$, which deals with the freedom of scalings.
Now, let $A \in \GL_n(\R)$, with $\abs{\det(A)} = 1$, be such that $AK$ is in Behrend position as well.
Note that this implies that $\D(AK) = 2$.
By \cref{lem_lown_isodiam}, both $K$ and $AK$ are in L\"owner position.
In particular, $K \subseteq A^{-1}\B^n_2$ and $\B^n_2$ is the unique minimal volume ellipsoid containing $K$.
Hence, $A^{-1} \B^n_2 = \B^n_2$ and thus $A$ is an orthogonal transformation.
\end{proof}

\subsection{The \texorpdfstring{$o$}{o}-symmetric reverse isodiametric inequality}

We conclude this section with a proof of \cref{conj_rev_idi} for $o$-symmetric convex bodies.

\begin{proof}[Proof of \cref{thm_o-symm_rev_isod_ineq}]
A crucial result of Barthe~\cite{barthe1998an} states that for every $L \in \K_o^n$ in L\"owner position, we have
\[
\V(L) \geq \frac{2^n}{n!},
\]
and equality holds if and only if $L$ is a regular crosspolytope.

Since $K$ is in Behrend position, \cref{thm_equival_isod_lown} implies that $(K-K)/\D(K) = (2/D(K))K$ is in L\"owner position.
Hence, we obtain
\[
\iq(K) = \frac{\V(K)}{\D(K)^n}=\frac{1}{2^n}\V\left(\frac{2}{\D(K)}K\right)\geq\frac{1}{n!},
\]
and the equality case characterization follows from that of Barthe.
\end{proof}

\section{The diametrical directions of a body in Behrend position are well-distributed}\label{sect_distribution_lemma}

In contrast to the $o$-symmetric situation, a complete solution of Makai Jr.'s conjecture for arbitrary convex bodies $K \in \K^n$ still seems to be elusive.
However, in the following we make significant progress on asymptotic bounds on the isodiametric quotient of a convex body in Behrend position.

As discussed in the introduction, Behrend obtained an optimal result in the plane.
Based on his ideas, we show that in isodiametric optimal position, the diametrical directions of a convex body give rise to a well-distributed point set on the sphere.
Once this distribution property is established, a strong asymptotic bound follows easily.

At the heart of Behrend's arguments lies his \emph{diameter condition}:

\begin{lemma}[{\cite[p.~733, Satz 9u]{behrend1937ueber}}]\label{lem_behrend}
For every $K\in\K^2$ in Behrend position, each closed right angular domain contains some diametrical direction of~$K$.
\end{lemma}

It turns out that Behrend's proof and therefore this property can be generalized to higher dimensions.
In order to state the extension, we define the angle between a linear subspace $L$ and a non-zero vector $v \in \R^n$ as
\[
\sphericalangle(L,v) = \min_{z \in L \setminus\{0\}}\left( \arccos{\frac{z^\intercal v}{\|z\| \|v\|}} \right).
\]

\begin{lemma}\label{ExtBehrend}
Let $K\in\K^n$ be in Behrend position, let $\D_K \subseteq \bS^{n-1}$ be the set of diametrical directions of~$K$, and let $1 \leq i \leq n-1$.
Then, for every $i$-dimensional linear subspace~$L$,
\begin{enumerate}[(i)]
 \item there is some $v\in \D_K$ such that $\sphericalangle(L,v)\leq\arccos(\sqrt{i/n})$, and
 \item there is some $w\in \D_K$ such that $\sphericalangle(L,w)\geq\arccos(\sqrt{i/n})$.
\end{enumerate}
Moreover, the cube $C_n = [-1,1]^n$ and the subspaces $L_i = \lin\{e_1,\dots,e_i\}$, where $e_i$ denotes the $i$th coordinate unit vector, show that the bounds cannot be improved.
\end{lemma}
\begin{proof}
(i): Let $L$ be a fixed $i$-dimensional linear subspace.
For the sake of contradiction, we assume that $\sphericalangle(L,v)>\arccos(\sqrt{i/n})$, for every $v\in \D_K$.
This means, that $|u^\intercal v|<\sqrt{i/n}$, for every $v\in \D_K$ and every $u\in L\cap\bS^{n-1}$.
Hence, a given diametrical direction $v\in \D_K$ encloses an angle
\[
\omega = \frac{\pi}{2} - \arccos(u^\intercal v) = \arcsin(u^\intercal v) < \arcsin\big(\sqrt{\tfrac{i}{n}}\big)
\]
with the $(n-i)$-plane $L^\perp$ orthogonal to $L$.
By $\cos(\arcsin(x))=\sqrt{1-x^2}$, this implies that $\cos^2\omega>(n-i)/n$.
Since $K$ is compact there exists a $\delta>0$ such that for every $v\in \D_K$ with corresponding angle $\omega$ we have
\begin{align}
\cos^2\omega&>\frac{n-i}{n}(1+\delta).\label{eqn_cos2_bound}
\end{align}
Via a suitable rotation of $K$, we assume that $L=\lin\{e_1,\ldots,e_i\}$.
For a small $\varepsilon>0$, we consider the linear map $A_\varepsilon = \diag(1,\ldots,1,1-\varepsilon,\ldots,1-\varepsilon)\in\GL_n(\R)$ having its first $i$ entries equal to $1$.
Using elementary trigonometry, we see that the length of a line segment $\ell$ that encloses an angle $\omega$ with~$L^\perp$, changes under the transformation $A_\varepsilon$ according to the formula
\begin{align}
\|A_\varepsilon \ell\| & =\|\ell\|\sqrt{1-2\varepsilon\cos^2\omega+\varepsilon^2\cos^2\omega} = \|\ell\|\left(1-\varepsilon\cos^2\omega+O(\varepsilon^2)\right).\label{eqn_line_segment_change}
\end{align}
Let $K'=A_\varepsilon K$.
By compactness of $K$, we can choose $\varepsilon$ small enough such that $\sphericalangle(L,v')>\arccos(\sqrt{i/n})$ for every diametrical direction $v' \in \D_{K'}$ of~$K'$ as well.
Also, every diametrical direction $v' \in \D_{K'}$ encloses an angle~$\omega$ with~$L^\perp$ also satisfying~\eqref{eqn_cos2_bound} for some $\delta > 0$.
Thus, if~$\ell$ is a line segment whose image under $A_\varepsilon$ attains $\D(K')$ and which encloses an angle of $\omega$ with $L^\perp$, we get by~\eqref{eqn_line_segment_change} that
\begin{align}
\D(K') & =\|\ell\|\left(1-\varepsilon\cos^2\omega+O(\varepsilon^2)\right) \leq \D(K)\left(1-\varepsilon\cos^2\omega+O(\varepsilon^2)\right).\label{eqn_diameter_change}
\end{align}
Clearly, we have $\V(K')=(1-\varepsilon)^{n-i}\V(K)$, and therefore for $\varepsilon$ small enough
\begin{align*}
\iq(K') &= \frac{\V(K')}{\D(K')^n}\overset{\eqref{eqn_diameter_change}}{\geq}\frac{\V(K)}{\D(K)^n}\frac{(1-\varepsilon)^{n-i}}{\left(1-\varepsilon\cos^2\omega+O(\varepsilon^2)\right)^n}\\
&=\iq(K)\,\frac{1-(n-i)\varepsilon+O(\varepsilon^2)}{1-n\varepsilon\cos^2\omega+O(\varepsilon^2)}\\
&\overset{\eqref{eqn_cos2_bound}}{>}\iq(K)\,\frac{1-(n-i)\varepsilon+O(\varepsilon^2)}{1-(n-i)\varepsilon(1+\delta)+O(\varepsilon^2)}\\
&> \iq(K).
\end{align*}
This is in contradiction that $K$ is in Behrend position and hence proves our claim.

(ii): The statement (ii) holds for the $i$-dimensional linear subspace $L$ if and only if (i) holds for its orthogonal complement~$L^\perp$.
Indeed, by (i) there exists some $w \in \D_K$ such that $\sphericalangle(L^\perp,w)\leq\arccos(\sqrt{(n-i)/n})$.
Therefore,
\[
\sphericalangle(L,w) = \frac{\pi}{2}-\sphericalangle(L^\perp,w) \geq \arcsin\big(\sqrt{\tfrac{n-i}{n}}\big) = \arccos\big(\sqrt{\tfrac{i}{n}}\big),
\]
in view of the identity $\arcsin(x)=\arccos(\sqrt{1-x^2})$.

We conclude the proof by showing that the cube $C_n = [-1,1]^n$ does not allow for a smaller angle than $\arccos(\sqrt{i/n})$ in (i).
First of all, $C_n$ is in L\"owner position (cf.~\cite[Sect.~2]{henk2012loewnerjohn}), and thus by \cref{thm_equival_isod_lown}, it is also in Behrend position.
The diametrical directions of $C_n$ are precisely its vertex directions.
For the linear subspace $L_i = \lin\{e_1,\dots,e_i\}$ and any vertex $v \in \{-1,1\}^n$ of $C_n$, we have
\[
\sphericalangle(L_i,v) = \min_{z \in L_i \setminus\{0\}}\left( \arccos{\frac{z^\intercal v}{\|z\| \|v\|}} \right) = \arccos\bigg(\frac{z_v^\intercal v}{\|z_v\| \sqrt{n}}\bigg) = \arccos\big(\sqrt{\tfrac{i}{n}}\big),
\]
where $z_v = (v_1,\ldots,v_i,0,\ldots,0)^\intercal$.
Hence, the inequalities in (i) and~(ii) cannot be improved in general.
\end{proof}

\begin{remark}\
\begin{enumerate}[(i)]
 \item Since $\arccos(\sqrt{1/2})=\pi/4$, we retrieve Behrend's diameter condition by \cref{ExtBehrend}~(ii), for $n=2$.

 \item For $u \in \bS^{n-1}$ and $\varphi \geq 0$, let $C(u,\varphi)=\{v \in \bS^{n-1} : \sphericalangle(u,v) \leq \varphi\}$ be the \emph{spherical cap} with center $u$ and angle $\varphi$.
 The case $i=1$ of \cref{ExtBehrend}~(i) then says that the caps of radius $\arccos(\sqrt{1/n})$ and with centers at the diametrical directions of $K$ induce a spherical covering, that is,
 \[
 \bS^{n-1} = \bigcup_{u \in \D_K} C(u,\arccos(\sqrt{1/n})).
 \]
\end{enumerate}
\end{remark}

A consequence of \cref{ExtBehrend} is \cref{thm_gen_rev_isod_ineq}, which is an asymptotic lower bound on the isodiametric quotient of a convex body in Behrend position that improves dramatically upon Makai Jr.'s original estimate~\eqref{eqn_Makai_estimate}.

\begin{proof}[Proof of \cref{thm_gen_rev_isod_ineq}]
The idea of the proof is to use \cref{ExtBehrend} in order to guarantee the existence of diametrical directions of~$K$ that span a simplex of large volume.

More precisely, let $v_1\in \D_K$ be chosen arbitrarily.
In view of \cref{ExtBehrend}~\romannumeral2), by induction, for every $1\leq i \leq n-1$, there exists a diametrical direction $v_{i+1} \in \D_K$ such that $\sphericalangle(L_i,v_{i+1}) \geq \arccos(\sqrt{i/n})$, where $L_i = \lin\{v_1,\ldots,v_i\}$.
By definition of $\D_K$, there are translation vectors $t_1,\ldots,t_n \in \R^n$ such that the segment $S_i = t_i+\D(K)[0,v_i]$ is contained in $K$, for $1\leq i \leq n$.
Clearly, the volume of~$K$ is then lower bounded by the volume of $\conv\{S_1,\ldots,S_n\}$.
A result of Groemer~\cite{groemer1966zusammen} (cf.~\cite[Thm.~2]{betkehenk1993approx}) says that this volume is minimal if the line segments have a common endpoint.
That is,
\begin{align*}
\V(K) &\geq \vol(\conv\{S_1,\ldots,S_n\}) \geq \vol(\conv\{S_1-t_1,\ldots,S_n-t_n\})\\
&= \frac{\D(K)^n}{n!}|\det(v_1,\ldots,v_n)|=\frac{\D(K)^n}{n!}\prod_{i=1}^{n-1}\sin(\sphericalangle(L_i,v_{i+1}))\\
&\geq\frac{\D(K)^n}{n!}\prod_{i=1}^{n-1}\sin(\arccos(\sqrt{i/n}))=\frac{\D(K)^n}{n!}\prod_{i=1}^{n-1}\sqrt{1-\frac{i}{n}}\\
&=\frac{\D(K)^n}{\sqrt{n!}\,n^{n/2}},
\end{align*}
where we also used that $\sin(\arccos(x))=\sqrt{1-x^2}$.
The asymptotics of this bound follow from Stirling's approximation of the factorial function.
\end{proof}

\section{Dvoretzky-Rogers-type volume bounds}\label{sect_DR_bounds}

Motivated by \cref{ExtBehrend} and its relevance to the reverse isodiametric problem, we investigate an extension of the famous Dvoretzky-Rogers lemma.

\begin{theorem}[Dvoretzky \& Rogers~\cite{dvoretzky1950Absolute}]
\label{thm:DRlemma}
Let $u_1,\ldots,u_m\in\bS^{n-1}$ and let $\lambda_1,\ldots,\lambda_m\geq 0$ be such that $\sum_{i=1}^m\lambda_i u_iu_i^\intercal=\I_n$.
Then there is a subset $\{u_{j_1},\ldots,u_{j_n}\} \subseteq \{u_1,\ldots,u_m\}$ such that
\[
\sphericalangle(L_i,u_{j_{i+1}}) \geq \arccos(\sqrt{i/n}), \quad \text{for} \quad 1 \leq i \leq n-1,
\]
where $L_i = \lin\{u_{j_1},\ldots,u_{j_i}\}$.
\end{theorem}

The characterization of the Behrend position in \cref{thm_equival_isod_lown} shows that \cref{ExtBehrend} is actually a generalization of this result in the sense that for \emph{every} $i$-dimensional linear subspace~$L$, there is an index $1 \leq j \leq m$ such that $\sphericalangle(L,u_j) \geq \arccos(\sqrt{i/n})$.

However, with regard to the reverse isodiametric problem, we are interested in finding a simplex $S = \conv\{0,u_{j_1},\dots,u_{j_n}\}$ that is spanned by a choice of the decomposing vectors $u_i$, and which has a large volume.
This motivates a more general Dvoretzky-Rogers-type problem derived from the following definition, which asks to find $j$-dimensional simplices of large volume in any decomposition of the identity.

\begin{definition}\label{Prob_dvorz_rog}
For $1 \leq j \leq n \leq m$, let $\DR(m,n,j)$ be the largest number $\nu \geq 0$ such that, for every $u_1,\ldots,u_m \in \bS^{n-1}$ and $\lambda_1,\ldots,\lambda_m \geq 0$ with $\I_n=\sum_{i=1}^m\lambda_i u_iu_i^\intercal$, there exist indices $1 \leq i_1 < \ldots < i_j \leq m$ fulfilling
\[
\vol_j(\conv\{0,u_{i_1},\dots,u_{i_j}\})\geq \nu.
\]
\end{definition}
\noindent A couple of remarks regarding the constants $\DR(m,n,j)$ are in order:
\begin{itemize}
 \item The constant $\DR(m,n,j)$ is non-increasing in $m$, because every decomposition of $\I_n$ into $m$ summands can be turned into one with $m+1$ summands.

 \item For every decomposition $\I_n=\sum_{i=1}^m\lambda_i u_iu_i^\intercal$ of the identity with $m \geq \binom{n+1}{2}$ vectors, there is a subset
$u_{i_1},\dots,u_{i_\ell}$, for some $\ell \leq \binom{n+1}{2}$, that also decomposes the identity (cf.~\cref{thm_lowner_position}).
 Hence
$\DR(m,n,j) = \DR(\binom{n+1}{2},n,j)$, for $m \geq \binom{n+1}{2}$.

 \item The proof of \cref{thm_gen_rev_isod_ineq} shows that the Dvoretzky-Rogers lemma implies
 \begin{align}
 \DR(m,n,n) \geq \frac{1}{\sqrt{n!}\,n^{n/2}},\label{eqn_DR_bound}
 \end{align}
 which is however not sensitive to the value of~$m$.
\end{itemize}

\noindent In the following, we use the classical Cauchy-Binet formula for the minors of a product of two matrices in order to provide estimates on $\DR(m,n,j)$ in terms of $m$, $n$, and~$j$.
The obtained bounds improve in particular the Dvoretzky-Rogers bound~\eqref{eqn_DR_bound} on $\DR(m,n,n)$ and they turn out to be sharp for interesting families of triples $(m,n,j)$.
For the sake of notation, we write $[n] = \{1,\ldots,n\}$ for the set of the first $n$ natural numbers, and $\binom{[n]}{i}$ for the family of all $i$-element subsets of $[n]$.
Given a matrix $M \in \R^{n\times m}$ and index sets $I \in \binom{[n]}{i}$ and $J \in \binom{[m]}{j}$, we denote by $M_{I,J}$ the submatrix of $M$ which remains after we delete all rows of $M$ with indices not in~$I$, and all columns of $M$ with indices not in~$J$.

\begin{theorem}[Cauchy-Binet formula, {cf.~\cite[Ch.~4]{broidawilliamson1989a}}]\label{thm_cauc_binet}
Let $A\in\R^{n\times m}, B\in\R^{m\times n}$, and let $I,J\in\binom{[n]}{i}$.
Then
\[
\det((AB)_{I,J})=\sum_{K\in\binom{[m]}{i}}\det(A_{I,K})\det(B_{K,J}).
\]
\end{theorem}

\noindent The next corollary was proven in \cite[Prop.~2.1]{pelczynski1991parallelepipeds} for $i=n$.

\begin{corollary}\label{cor_cauc_binet}
Let $u_1,\ldots,u_m\in\bS^{n-1}$ and $\lambda_1,\ldots,\lambda_m\geq 0$ be such that $\sum_{i=1}^m\lambda_i u_iu_i^\intercal=\I_n$.
Then, for every $1\leq i \leq n$, we have
\[
\binom{n}{i}=\sum_{J\in\binom{[m]}{i}}\lambda_J\det((U_J)^\intercal U_J),
\]
where $\lambda_J=\prod_{j\in J}\lambda_j$ and $U_J=(u_j:j\in J) \in \R^{n \times i}$.
\end{corollary}
\begin{proof}
For $1 \leq i \leq n$, let $w_i = \sqrt{\lambda_i}u_i$, and write $W = (w_1,\ldots,w_m) \in \R^{n\times m}$.
First, we show that $\I_n=\sum_{i=1}^mw_iw_i^\intercal=WW^\intercal$.
The first identity follows from the definition of the $w_i$, whereas the second follows from
\begin{align*}
\left<\e_k,\left(\sum_{i=1}^m w_iw_i^\intercal\right)\e_l\right> & = \sum_{i=1}^m\left<\e_k,w_iw_i^\intercal\e_l\right> = \sum_{i=1}^m \left<\e_k,\left<w_i,\e_l\right>w_i\right> \\
& = \sum_{i=1}^m\left<\e_k,w_i\right>\left<w_i,\e_l\right>=\sum_{i=1}^m w_{il}w_{ik}=\left<\e_k,WW^\intercal\e_l\right>,
\end{align*}
where we have used the alternative notation $\left<x,y\right> = x^\intercal y$ for the standard scalar product to improve readability.
Now, for every $I,J \in \binom{[n]}{i}$, let $\delta_{I,J} = 1$, if $I=J$, and $\delta_{I,J} = 0$, otherwise.
\cref{thm_cauc_binet} then implies
\begin{align*}
\delta_{I,J} = \det((\I_n)_{I,J}) &= \det((WW^\intercal)_{I,J})\\
&= \sum_{K\in\binom{[m]}{i}}\det(W_{I,K})\det((W^\intercal)_{K,J}).
\end{align*}
Therefore, using \cref{thm_cauc_binet} once again, we arrive at
\begin{align*}
\binom{n}{i}&=\sum_{I\in\binom{[n]}{i}}\sum_{J\in\binom{[m]}{i}}\det(W_{I,J})\det((W^\intercal)_{J,I})\\
&=\sum_{J\in\binom{[m]}{i}}\sum_{I\in\binom{[n]}{i}}\det((W^\intercal)_{J,I})\det(W_{I,J})=\sum_{J\in\binom{[m]}{i}}\det((W^\intercal W)_{J,J})\\
&=\sum_{J\in\binom{[m]}{i}}\det((W_J)^\intercal W_J)=\sum_{J\in\binom{[m]}{i}}\lambda_J\det((U_J)^\intercal U_J).\qedhere
\end{align*}
\end{proof}

The next result, for all $\lambda_i$ positive, is contained in~\cite[Ch.~II, \S 2.22, \bf 52]{hardylittlewoodpolya1952inequalities}, and in~\cite[Ch.~1, \S 12, (8)]{beckenbachbellman1965inequalities}, from which there follows the same inequality for all $\lambda_i \geq 0$.
However, the equality case is not discussed in either of these books, therefore for completeness we give a simple proof, clarifying also the cases of equality.

\begin{lemma}\label{lem:max_sum_prod_lambdas}
For $m\in \N$ and $c \geq 0$, let $\Delta_m^c = \left\{\lambda \in \R^m_{\geq 0} : \sum_{i=1}^m \lambda_i = c\right\}$.
For $d \in \N$ with $d \leq m$, let $\sigma_d(\lambda_1,\ldots,\lambda_m) = \sum_{I \in \binom{[m]}{d}}\prod_{i \in I}\lambda_i$ be the $d$-th elementary symmetric polynomial.
Then
\[
\max_{\lambda \in \Delta_m^c} \sigma_d(\lambda) = \sigma_d\left(\frac{c}{m},\ldots,\frac{c}{m}\right) = \binom{m}{d}\frac{c^d}{m^d},
\]
and this maximum is attained only for $\lambda_1 = \ldots = \lambda_m = c/m$.
\end{lemma}
\begin{proof}
Let us first assume that if $\lambda \in \Delta_m^c$ attains the maximum value $\sigma_d(\lambda)$, then $\lambda_i \neq 0$, for all $1\leq i \leq m$.
In this case we use an indirect argument and suppose that the maximum is attained at some $\lambda \in \Delta_m^c$ with $\lambda_i > \lambda_j$, for some $1\leq i < j \leq m$.
Let $\eps > 0$ be small enough such that $\lambda_i - \lambda _j > \eps$, and let
$\bar \lambda = \lambda - \eps e_i + \eps e_j$.
For a subset $J \subseteq [m]$, we write $\lambda_J = \prod_{j \in J}\lambda_j$.
From the definition it follows that $\bar \lambda \in \Delta_m^c$, and moreover we have
\begin{align*}
\sigma_d(\bar \lambda) &= \sum_{i \in I \in \binom{[m]\setminus\{j\}}{d}}\bar\lambda_I + \sum_{j \in I \in \binom{[m]\setminus\{i\}}{d}}\bar\lambda_I + \sum_{\{i,j\} \subseteq I \in \binom{[m]}{d}}\bar\lambda_I + \sum_{I \in \binom{[m]\setminus\{i,j\}}{d}}\bar\lambda_I\\
&= (\lambda_i - \eps)\sum_{J \in \binom{[m]\setminus\{i,j\}}{d-1}}\lambda_J + (\lambda_j + \eps)\sum_{J \in \binom{[m]\setminus\{i,j\}}{d-1}}\lambda_J\\
&\phantom{=} + (\lambda_i - \eps)(\lambda_j + \eps)\sum_{J \in \binom{[m]\setminus\{i,j\}}{d-2}}\lambda_J + \sum_{I \in \binom{[m]\setminus\{i,j\}}{d}}\lambda_I\\
&= \sigma_d(\lambda) + \eps(\lambda_i - \lambda_j - \eps)\sum_{J \in \binom{[m]\setminus\{i,j\}}{d-2}}\lambda_J\\
&> \sigma_d(\lambda),
\end{align*}
contradicting the maximality of $\lambda$.
Note, that the inequality in the last line above is strict, because all $\lambda_i$ are assumed to be positive.

In order to finish the proof, we now show by induction on $m$, that in every optimal solution $\lambda \in \Delta_m^c$ we have $\lambda_i \neq 0$, for all $1\leq i \leq m$.
For $m=d$, we have $\sigma_d(\lambda) = \prod_{i=1}^d\lambda_i$.
So, if one of the $\lambda_i$ would vanish, then clearly the point is not a maximum.
If $m > d$, and without loss of generality $\lambda_1 \geq \ldots \geq \lambda_k > \lambda_{k+1} = \ldots = \lambda_m = 0$, then in view of what was shown above 
\[
\sigma_d(\lambda_1,\ldots,\lambda_m) = \sigma_d(\lambda_1,\ldots,\lambda_k) \leq \binom{k}{d}\frac{c^d}{k^d} < \binom{m}{d}\frac{c^d}{m^d}.
\]
Hence, $\lambda$ could not have been an optimum.
\end{proof}

We are now prepared to give our estimates on the Dvoretzky-Rogers-type constants $\DR(m,n,j)$.

%
%
%

\begin{proof}[Proof of \cref{thm_improved_Dvor_rog_ineq}]
Let $1 \leq j \leq n \leq m \leq \binom{n+1}{2}$, and let $u_1,\ldots,u_m\in\bS^{n-1}$ and $\lambda_1,\ldots,\lambda_m\geq 0$ be such that $\sum_{i=1}^m\lambda_i u_iu_i^\intercal=\I_n$.
Let the elements of a subset $J \in \binom{[m]}{j}$ be indexed by $J=\{i_1,\ldots,i_j\}$, and let $S_J = \conv\{0,u_{i_1},\ldots,u_{i_j}\}$ be the corresponding simplex.
With this notation, \cref{cor_cauc_binet} gives us
\begin{align*}
\binom{n}{j} & =\sum_{J \in \binom{[m]}{j}}\lambda_J\det((U_J)^\intercal U_J) = \sum_{J \in \binom{[m]}{j}}\lambda_J\left(j!\,\V_j(S_J)\right)^2 \\
&\leq (j!)^2 \sum_{J \in \binom{[m]}{j}}\lambda_J \bigg(\max_{J \in \binom{[m]}{j}}\V_j(S_J)\bigg)^2.
\end{align*}
By taking traces in $\I_n=\sum_{i=1}^m \lambda_i u_iu_i^\intercal$ and using $\|u_i\|=1$, we see that $n=\sum_{i=1}^m\lambda_i$.
Thus, we can apply \cref{lem:max_sum_prod_lambdas}, and obtain that
\[
\sum_{J \in \binom{[m]}{j}}\lambda_J \leq \binom{m}{j}\left(\frac{n}{m}\right)^j.
\]
Continuing the previous estimate we therefore arrive at
\[
\binom{n}{j} \leq (j!)^2 \binom{m}{j}\left(\frac{n}{m}\right)^j\bigg(\max_{J \in \binom{[m]}{j}}\V_j(S_J)\bigg)^2,
\]
as desired.

Let us now discuss equality cases for certain triples of parameters.
From the proof of the inequalities above we see that the bound on $\DR(m,n,j)$ is tight if and only if there is a decomposition $\I_n = \sum_{i=1}^m \lambda_iu_iu_i^\intercal$ such that
\begin{enumerate}[(i)]
 \item $\V_j(S_J) = \V_j(S_{J'})$, for every $J,J' \in \binom{[m]}{j}$ , and
 \item $\lambda_1 = \ldots = \lambda_m = \frac{n}{m}$ (see~\cref{lem:max_sum_prod_lambdas}).
\end{enumerate}
First of all, for every $1 \leq j \leq n$, we have
\[
\DR(n,n,j) = \frac{1}{j!}.
\]
In fact, if $\pm u_1,\ldots,\pm u_n$ are the vertices of a regular crosspolytope, then $\V_j(S_J) = 1 / j!$,
for every $J \in \binom{[n]}{j}$, and $\I_n = \sum_{i=1}^n u_iu_i^\intercal$.

Second, for every $1 \leq j \leq n$, we have
\[
\DR(n+1,n,j)^2 = \frac{(n-j+1)(n+1)^{j-1}}{n^j(j!)^2}.
\]
Indeed, if $u_1,\ldots,u_{n+1}$ are the vertices of a regular simplex, then every~$j$ of these vertices give rise to a $j$-dimensional simplex with the same volume.
Moreover, the reader quickly convinces herself that the coefficients $\lambda_i$ in the corresponding decomposition of the identity matrix are all equal to $n/(n+1)$.

Finally, we consider the case $j=2$ and $m = \binom{n+1}{2}$.
Writing $J = \{\ell,k\}$, we get
\[
\det((U_J)^\intercal U_J) = \det\left(\begin{array}{cc}1&u_\ell^\intercal u_k\\u_\ell^\intercal u_k&1\end{array}\right) = 1-(u_\ell^\intercal u_k)^2.
\]
Hence, the triangles $\V_2(S_J)$, $J \in \binom{[m]}{2}$, all have the same volume if
\[
\abs{\cos(\sphericalangle(u_\ell,u_k))} = |u_\ell^\intercal u_k| = \sqrt{1-\frac{\binom{n}{2}\left(\frac{m}{n}\right)^2}{\binom{m}{2}}} = \frac{1}{\sqrt{n+2}},
\]
for every $1 \leq \ell < k \leq m = \binom{n+1}{2}$.
In other words, the vectors $u_i$ are the directions of a set of~$\binom{n+1}{2}$ \emph{equiangular lines}.
(By a result of M.~Gerzon, cf.~\cite[Thm.~3.5]{lemmens1973equiangular} and its proof, the maximal number of equiangular lines in $\R^n$ is at most $\binom{n+1}{2}$, and this many equiangular lines necessarily enclose angles equal to $\arccos(1/\sqrt{n+2})$.)
For $n=2$, the directions from the center of an equilateral triangle to its vertices give a system of three equiangular lines.
For $n=3$, we may take the directions from the center of a regular icosahedron to its vertices.
In dimensions $n=7$ and $n=23$, there exist sets of $28$ and $276$ equiangular lines with an angle of $\arccos(1/3)$ and $\arccos(1/5)$, respectively, and for $n\leq 118$, $n \neq 2,3,7,23$, no such configuration exists (see \cref{rem_DR_mnj}~(\romannumeral3) for further details).
\end{proof}

\begin{remark}\label{rem_DR_mnj}\
\begin{enumerate}[(i)]
\item The technique used for proving \cref{thm_improved_Dvor_rog_ineq} is an extension of the arguments used by Pe\l{c}zy\'{n}ski \& Szarek
\cite[Prop.~2.1]{pelczynski1991parallelepipeds}, who considered the case $\DR(m,n,n)$. Moreover, this same bound has been obtained recently with probabilistic methods by Fodor, Nasz\'{o}di \& Zarn\'{o}cz~\cite{fodornaszoditamas2018on}.
    They also illustrate that the bound on $\DR(n+1,n,n)$ is tight because of the regular simplex.

\item \cref{thm_equival_isod_lown}~(v) allows us to think of $\DR(m,n,n)$ as the solution of a polynomial optimization problem.
Using the \texttt{scip} solver~\cite{MaherFischerGallyetal.2017}, we have obtained numerical evidence that
\[
\DR(5,3,3)=\frac1{8}\quad\text{and}\quad\DR(6,3,3)=\frac1{6\sqrt{2}}.
\]
Together with the proven values $\DR(3,3,3)=1/6$ and $\DR(4,3,3)=2/(9\sqrt{3})$ (cf.~\cref{thm_improved_Dvor_rog_ineq}),
this would completely solve the determination of $\DR(m,n,j)$ from \cref{Prob_dvorz_rog} in the case $n=j=3$.
The experiments with \texttt{scip} further suggest that $\DR(5,3,3)$ is attained by
\[
(u_1,\ldots,u_5) = \left(\begin{array}{ccccc}
0 & \frac{\sqrt{3}}{2} & -\frac{\sqrt{3}}{2} & 0 & 0 \\
0 & 0 & 0 & \frac{\sqrt{3}}{2} & -\frac{\sqrt{3}}{2}\\
1 & \frac12 & \frac12 & -\frac12 & -\frac12
\end{array}\right)
\]
and $(\lambda_1,\ldots,\lambda_5) = (\frac13,\frac23,\ldots,\frac23)$.

\item To date, $n=2,3,7,23$ are the only dimensions in which we know that there exist configurations of sets of $\binom{n+1}{2}$ equiangular lines of angle $\arccos\left(1/\sqrt{n+2}\right)$.
These configurations are constructed in~\cite{lemmens1973equiangular}.
Moreover, a result of Neumann (cf.~\cite[Thm.~3.4]{lemmens1973equiangular}) states that, if for $n \geq 4$ there is a set of $\binom{n+1}{2}$ equiangular lines, then $\sqrt{n+2}$ is an odd integer.
By results of Bannai et al.~\cite{bannaimunembenkov2005the} the first unknown candidate is $n=119$.
In general, proving the existence of such configurations in dimension $n$, where $n+2=k^2$ for some odd $k \geq 11$, would show the tightness of \cref{thm_improved_Dvor_rog_ineq} for $\DR(\binom{n+1}{2},n,2)$.
\end{enumerate}
\end{remark}


As a corollary to \cref{thm_improved_Dvor_rog_ineq}, we get an alternative proof of the asymptotic bound in \cref{thm_gen_rev_isod_ineq}.

\begin{corollary}\label{cor:asymptoticBound_viaCB}
Let $K\in\K^n$ be in Behrend position. Then,
\[
\iq(K) \geq \DR(\tbinom{n+1}{2},n,n) \geq \frac{\big(\frac{n+1}{2}\big)^{n/2}}{\binom{n(n+1)/2}{n}^{1/2} n!} \sim \sqrt{e} \left(\frac{2\pi}{n}\right)^{1/4} \frac{\sqrt{n+1}}{n!\,e^{n/2}}.
\]
For $n=2$, this is an alternative to Behrend's solution of the reverse isodiametric problem in the plane.
For $3$-dimensional bodies $K \in \K^3$ in Behrend position, it gives a lower bound of $\iq(K) \geq \sqrt{10}/30 \approx 0.10541$, which is very close to the conjectured optimal value $1/(6\sqrt{2}) \approx 0.11785$ in \cref{conj_rev_idi}.
\end{corollary}
\begin{proof}
Since $K$ is in Behrend position, \cref{thm_equival_isod_lown}~(v) provides us with diametrical directions $u_1,\ldots,u_m \in \D_K$ and scalars $\lambda_1,\ldots,\lambda_m \geq 0$, for some $n \leq m \leq \binom{n+1}{2}$, such that $\I_n = \sum_{i=1}^m \lambda_i u_iu_i^\intercal$.

By the same reasoning as in the proof of \cref{thm_gen_rev_isod_ineq}, for every choice of indices $1 \leq i_1 < \ldots < i_n \leq m$, we have
\[
\iq(K) \geq \V(\conv\{0,u_{i_1},\ldots,u_{i_n}\}).
\]
The definition of $\DR(m,n,j)$ then implies that $\iq(K) \geq \DR(\tbinom{n+1}{2},n,n)$ and so we can employ the lower bound from \cref{thm_improved_Dvor_rog_ineq}.
The asymptotics follow from Stirling's approximation of the factorial function.
\end{proof}

Therefore, the triple $(\binom{n+1}{2},n,n)$ is the most interesting concerning the reverse isodiametric problem.
In fact, the proof of \cref{cor:asymptoticBound_viaCB} shows that the following claim would imply Makai Jr.'s \cref{conj_rev_idi}~\eqref{eqn_conj_gen_case}.

\begin{conjecture}
For every $n \in \N$, we have $\DR(\binom{n+1}{2},n,n) = \frac{\sqrt{n+1}}{n!\,2^{n/2}}$.
\end{conjecture}

\section{Analogies to the reverse isominwidth problem}\label{sect_isominwidth_position}

For a convex body $K\in\K^n$ and a direction $u\in\R^n\setminus\{0\}$, the support function of $K$ with respect to~$u$ is defined as $h(K,u) = \max\{x^\intercal u:x\in K\}$.
The width of $K$ in direction $u\in\bS^{n-1}$ is given by $\w(K,u) = h(K,u)+h(K,-u)$.
Finally, the \emph{minimum width} of $K$ is defined as
\[
\w(K) = \min_{u\in\bS^{n-1}} \w(K,u).
\]
For an $o$-symmetric convex body $K \in \K^n_o$, the minimum width and the diameter are dual to each other in the sense that
\[
\w(K)\D(K^\star) = 4
\]
(cf.~\cite[(1.2)]{gritzklee1992inner}).
Here, $K^\star = \{x \in \R^n : x^\intercal y \leq 1,\,\forall y \in K\}$ denotes the \emph{polar body} of~$K$.
In the following, we elaborate on this duality and investigate the dual of the reverse isodiametric problem.

In analogy to the isodiametric quotient, we define the \emph{isominwidth quotient} of a convex body $K \in \K^n$ as
\[
\iwq(K):=\frac{\V(K)}{\w(K)^n}
\]
and we may ask for upper and lower bounds on this magnitude.
The question on optimal lower bounds is classical in Convex Geometry.
P\'al \cite{pal1921ein} proved that, for every planar $K\in\K^2$, we have
\[
\iwq(K)\geq\frac{1}{\sqrt{3}},
\]
and that equality holds if and only if $K$ is an equilateral triangle.
In arbitrary dimension, the following bound is due to
Firey~\cite{firey1965lower} (see also K.~Bezdek~\cite{bezdek2013tarski} for a slightly improved yet much more involved bound):
\[
\iwq(K)\geq\frac{2}{\sqrt{3}n!},\quad\text{for }K\in\K^n.
\]
However, the optimal bound (often called \emph{the convex Kakeya problem} or \emph{P\'al problem}) is not known.
Already in $\R^3$, one can slice a small neighbourhood of a vertex of a regular tetrahedron $T_3$, obtaining a new polytope $T_3'$, without reducing its minimum width.
Hence, one gets $\iwq(T_3')<\iwq(T_3)$, so that~$T_3$ is not a minimizer in P\'al's problem.
If one continues slicing $T_3'$ in a certain way until no more slicing is possible without reducing the minimum width, Heil~\cite{heil1978kleinste} conjectures that the resulting body is the solution to P\'al's question.

If we restrict to $o$-symmetric convex bodies $K\in\K^n_o$ the situation gets much easier.
Indeed, since $(\w(K)/2)\B^n_2\subseteq K$, one obtains
\[
\iwq(K)\geq\frac{\V(\B^n_2)}{2^n},
\]
which holds with equality if and only if $K$ is a Euclidean ball.

Analogously to lower estimates on the isodiametric quotient, there exists no upper bound on $\iwq(K)$ that is independent of the body $K \in \K^n$.
Hence, we may study whether the minimal isominwidth quotient among all linear images of~$K$ can be upper bounded by a constant only depending on the dimension~$n$.

First of all, an analogous argumentation as in \cref{lem_attainment_Behrendposition} leads to

\begin{lemma}\label{thm_isominwidth_position}
For every $K \in \K^n$, there exists an $A \in \GL_n(\R)$ such that
\[
\iwq(AK) = \inf_{B \in \GL_n(\R)}\iwq(BK).
\]
\end{lemma}

This of course allows to define yet another \emph{position} of a convex body, this time with respect to the minimum width.

\begin{definition}\label{def_isominwidth_position}
A convex body $K\in\K^n$ is in \emph{isominwidth position}, if
\[
\iwq(K) = \min_{A\in\GL_n(\R)} \iwq(AK).
\]
\end{definition}

\noindent Now, we want to establish the analog of \cref{thm_equival_isod_lown} for the isominwidth position.
To this end, we need some further notation.
Let
\[
\wD_K = \left\{ u \in \bS^{n-1} : \, h(K,u)+h(K,-u)=\w(K) \right\}
\]
be the set of \emph{minwidth directions}, that is, the directions in which the minimum width of $K$ is attained.
It is well-known that if $u\in\wD_K$, then there exists an $x\in K$ such that $x+\w(K)[0,u]\subseteq K$ (cf.~\cite{gritzklee1992inner}).
Moreover, let $H(K,u)=\{x \in \R^n : x^\intercal u =  h(K,u)\}$ be the supporting hyperplane of $K$ in the direction $u$, and let $H^-(K,u)$ be the corresponding halfspace containing~$K$.

Just as the Behrend position is strongly tied to the L\"owner position, it turns out that the isominwidth position is linked to the John position of a convex body.
Dually to the L\"owner position, $K\in\K^n$ is in \emph{John position} if~$\B^n_2$ is the maximum volume ellipsoid, called the \emph{John ellipsoid}, contained in~$K$.
The charaterization of the John position by the existence of a certain decomposition of the identity is verbatim to \cref{thm_lowner_position}, except for that we need to replace the condition $K \subseteq \B^n_2$ by $\B^n_2\subseteq K$ (cf.~\cite{henk2012loewnerjohn} and \cite[Ch.~11]{gruber2007convex}).
We now formulate our desired characterization of the isominwidth position.

\begin{theorem}\label{thm_isominwidth_John}
Let $K\in\K^n$. The following are equivalent:
\begin{enumerate}[(i)]
 \item $K$ is in isominwidth position.
 \item $K-K$ is in isominwidth position.
 \item $(K-K)/\w(K)$ is in John position.
 \item $\bigcap_{u\in\wD_K}H^-(\B^n_2,u)$ is in John position.
 \item There exists an $m\in\{n,\dots,\binom{n+1}{2}\}$, scalars $\lambda_1,\ldots,\lambda_m\geq 0$, and minwidth directions $u_1,\ldots,u_m \in \wD_K$, such that
\[
\I_n=\sum_{i=1}^m\lambda_i u_iu_i^\intercal.
\]
\end{enumerate}
\end{theorem}

The proof of this characterization is based on the same ideas as that for the Behrend position given in \cref{sect_behrend_loewner}.
For the sake of brevity, we do not give the details here and leave them to the reader.

Just like the L\"owner ellipsoid, also the John ellipsoid of a convex body~$K$ is unique (cf.~\cite{artstein2015asymptotic,henk2012loewnerjohn}).
For an $o$-symmetric convex body $K$, both the L\"owner and the John ellipsoid are $o$-symmetric as well.
Moreover, since for every ellipsoid $E \in \K^n$ we have $\vol(E) \vol(E^\star) = \vol(\B^n_2)^2$, we see that the polar of the L\"owner ellipsoid of $K$ is the John ellipsoid of $K^\star$, and vice versa.
Again by uniqueness of the John ellipsoid, and analogously to \cref{thm_uniqueness_Behrend_pos}, we obtain

\begin{proposition}\label{thm_uniqueness_isominwidth_pos}
The isominwidth position of a convex body is unique up to orthogonal transformations, scalings, and translations.
\end{proposition}

As the main result of this section, we completely solve the reverse iso\-minwidth problem.
Curiously, it turns out that $o$-symmetric convex bodies have the worst minimum isominwidth quotient, which is in strong contrast to the Behrend position.
The proof follows the ideas developed by Ball~\cite{ball1991volume} for the volume ratio of a convex body in John position.
We need Ball's~\cite[Lem.~5]{ball1991volume} geometric version of the inequality of Brascamp \& Lieb (cf.~\cite[Cor.~3, Prop.~1, Thms.~4 and 5]{barthe1998onareverse}).
Notice that the case of equality of Corollary~3 in~\cite{barthe1998onareverse} is stated incorrectly.
Therefore, we here explicitly state the correct form, which in fact follows from the considerations in~\cite[Sect~2.3.1, pp.~352--353]{barthe1998onareverse}.

\begin{theorem}[Brascamp \& Lieb 1976, Ball 1991, Barthe 1998]\label{thm_BL_ineq}\ \\
Let $\lambda_1,\ldots,\lambda_m > 0$ and $u_1,\ldots,u_m \in \bS^{n-1}$, 
for some $n \leq m \in \N$, be such that $\sum_{i=1}^m\lambda_i u_iu_i^\intercal=\I_n$.
Further, let $f_1,\ldots,f_m:\R\rightarrow[0,\infty)$ be measurable functions with $0<\|f_i\|_1<\infty$, $i=1,\dots,m$.
Then
\[
\int_{\R^n}\bigg(\prod_{i=1}^mf_i(x^\intercal u_i)^{\lambda_i}\bigg)dx \leq \prod_{i=1}^m\left(\int_{\R}f_i(t)dt\right)^{\lambda_i}.
\]
Moreover, if none of the functions $f_i$ is a Gaussian function, then
equality holds if and only if
there is an orthonormal basis $\{v_1,\dots,v_n\}$ of $\R^n$, such that $\{v_1,\dots,v_n\}\subseteq\{u_1,\dots,u_m\}\subseteq\{\pm v_1,\dots,\pm v_n\}$,
the sum of the $\lambda_i$'s associated to the $u_i$'s which are equal to $v_k$ or $-v_k$ equals $1$, for each $1\leq k\leq n$,
the functions $f_i$, associated to the $u_i$'s which are equal to $v_k$ (resp.~$-v_k$), are proportional, for each $1\leq k\leq n$,
and if $u_i=v_k$ and $u_j=-v_k$, for some $i,j,k$, then $f_i(t)$ is proportional to $f_j(-t)$, for each $1\leq k\leq m$.
\end{theorem}

As said above, some equality cases are missing in \cite[Cor.~3]{barthe1998onareverse}.
For instance, notice that the set $\{e_1,-e_1\}$ in $\R$ gives a decomposition of the identity with $\lambda_1=\lambda_2=1/2$,
$\{e_1,-e_1\}$ is (in the notation of \cite{barthe1998onareverse}) irreducible, and if $f_1(t)=f_2(-t)$, we would have equality in \cref{thm_BL_ineq}.


\begin{proof}[Proof of \cref{thm_isominwidth_reverse_ineq}]
Applying a suitable scaling of $K$ we may suppose that $\w(K)=2$.
Since $K\subseteq H^-(K,u)$, for every $u\in\wD_K$, we have
\[
K\subseteq C = \bigcap_{u\in\wD_K}H^-(K,u).
\]
By \cref{thm_isominwidth_John}, we have $\I_n=\sum_{i=1}^m\lambda_i u_iu_i^\intercal$ for some $\lambda_i > 0$ and $u_i \in \wD_K$.

We define $r_i=h(K,-u_i)$, for $1 \leq i \leq m$, and we observe that $h(K,u_i) = \w(K,u_i) - r_i = 2 - r_i$.
Further, let
\[
f_i(t) = \chi_{[-r_i,-r_i+2]}(t) = \begin{cases} 1 & \text{if }-r_i\leq t\leq -r_i+2,\\ 0 & \text{otherwise,}\end{cases}
\]
be the characteristic function of $[-r_i,-r_i+2]$, and observe that
\[
C = \bigg\{x\in\R^n : \prod_{i=1}^m f_i(x^\intercal u_i)^{\lambda_i} = 1\bigg\}.
\]
Recalling from the proof of \cref{thm_improved_Dvor_rog_ineq} that $\sum_{i=1}^m \lambda_i = n$, \cref{thm_BL_ineq} yields that
\begin{align*}
\V(K) & \leq \V(C) = \int_{\R^n}\chi_C(x)dx = \int_{\R^n}\bigg(\prod_{i=1}^mf_i(x^\intercal u_i)^{\lambda_i}\bigg)dx \\
& \leq \prod_{i=1}^m\left(\int_{\R}f_i(t)dt\right)^{\lambda_i}=\prod_{i=1}^m2^{\lambda_i}=2^{\sum_{i=1}^m\lambda_i}=2^n = \w(K)^n,
\end{align*}
as desired.

If we have equality, we need to have equality in each step of the estimate above.
Hence, equality holds if and only if $\V(K)=\V(C)$, and thus $K=C$,
and, since none of the characteristic functions $f_i$ is a Gaussian function, equality in \cref{thm_BL_ineq} implies that
there exists an orthonormal basis $\{v_1,\dots,v_n\}$ of $\R^n$ such that
$\{v_1,\dots,v_n\}\subseteq\{u_1,\dots,u_m\}\subseteq\{\pm v_1,\dots,\pm v_n\}$. Moreover,
notice that if there exists $v_k$ such that $u_i=u_j=v_k$ (resp.~such that $u_i=v_k$ and $u_j=-v_k$),
then $f_i(t)=f_j(t)$ (resp.~$f_i(t)=\chi_{[-r_i,-r_i+2]}(t)=\chi_{[-r_j,-r_j+2]}(-t)=f_j(-t)$, where $r_i+r_j=h(K,-v_k)+h(K,v_k)=2$),
which is compatible with the equality cases of \cref{thm_BL_ineq}. In particular, $C$ is a cube of edge-length $2$ and whose edges are parallel to $v_1,\dots,v_n$, which concludes the proof.
\end{proof} 

\subsection*{Acknowledgments}
We thank Ambros Gleixner, Benjamin M\"uller, and Felipe Serrano from the Zuse Institute Berlin for discussions and help regarding the \texttt{scip} experiments described in \cref{rem_DR_mnj}.
Furthermore, we thank Peter Gritzmann and Martin Henk for the possibility of mutual research visits at TU Munich and TU Berlin.
The first author also thanks TU Munich and the University Centre of Defence of San Javier, where
part of this work has been done. We would like to thank the anonymous referee for comments and suggestions that helped us to improve the writing, and in particular, for pointing out that \cite[Cor.~3]{barthe1998onareverse} is incorrectly stated.

\bibliographystyle{amsplain}
\bibliography{mybib}

\providecommand{\bysame}{\leavevmode\hbox to3em{\hrulefill}\thinspace}
\providecommand{\MR}{\relax\ifhmode\unskip\space\fi MR }
\providecommand{\MRhref}[2]{%
  \href{http://www.ams.org/mathscinet-getitem?mr=#1}{#2}
}
\providecommand{\href}[2]{#2}
\begin{thebibliography}{10}

\bibitem{artstein2015asymptotic}
Shiri Artstein-Avidan, Apostolos Giannopoulos, and Vitali~D. Milman,
  \emph{Asymptotic geometric analysis. {P}art {I}}, Mathematical Surveys and
  Monographs, vol. 202, American Mathematical Society, Providence, RI, 2015.

\bibitem{ball1991volume}
Keith Ball, \emph{Volume ratios and a reverse isoperimetric inequality}, J.
  London Math. Soc. (2) \textbf{44} (1991), no.~2, 351--359.

\bibitem{ball1992ellipsoids}
\bysame, \emph{Ellipsoids of maximal volume in convex bodies}, Geom. Dedicata
  \textbf{41} (1992), no.~2, 241--250.

\bibitem{bannaimunembenkov2005the}
Eiichi Bannai, Akihiro Munemasa, and Boris Venkov, \emph{The nonexistence of
  certain tight spherical designs}, St. Petersburg Math. J. \textbf{16} (2005),
  no.~1, 609--625.

\bibitem{barthe1998an}
Franck Barthe, \emph{An extremal property of the mean width of the simplex},
  Math. Ann. \textbf{310} (1998), no.~4, 685--693.

\bibitem{barthe1998onareverse}
\bysame, \emph{On a reverse form of the {B}rascamp-{L}ieb inequality}, Invent.
  Math. \textbf{134} (1998), no.~2, 335--361.

\bibitem{beckenbachbellman1965inequalities}
Edwin~F. Beckenbach and Richard Bellman, \emph{Inequalities}, Second revised
  printing. Ergebnisse der Mathematik und ihrer Grenzgebiete. Neue Folge, Band
  30, Springer-Verlag, New York, Inc., 1965.

\bibitem{behrend1937ueber}
Felix Behrend, \emph{\"{U}ber einige {A}ffininvarianten konvexer {B}ereiche},
  Math. Ann. \textbf{113} (1937), no.~1, 713--747.

\bibitem{betkehenk1993approx}
Ulrich Betke and Martin Henk, \emph{Approximating the volume of convex bodies},
  Discrete Comput. Geom. \textbf{10} (1993), no.~1, 15--21.

\bibitem{bezdek2013tarski}
K\'aroly Bezdek, \emph{Tarski’s plank problem revisited}, Geometry —
  Intuitive, Discrete, and Convex. In: I.~Bárány, K.J.~Böröczky, G.~Fejes
  Tóth, J.~Pach (eds), Springer, Berlin, Heidelberg, 2013, pp.~45--64.

\bibitem{bieberbach1915ueber}
Ludwig Bieberbach, \emph{\"{U}ber eine {E}xtremaleigenschaft des {K}reises},
  Jber. Deutsch. Math.-Verein. \textbf{24} (1915), 247--250.

\bibitem{broidawilliamson1989a}
Joel~G. Broida and S.~Gill Williamson, \emph{A comprehensive introduction to
  linear algebra}, Addison-Wesley Publishing Company, Advanced Book Program,
  Redwood City, CA, 1989.

\bibitem{dvoretzky1950Absolute}
Aryeh Dvoretzky and C.~Ambrose Rogers, \emph{Absolute and unconditional
  convergence in normed linear spaces}, Proc. Natl. Acad. Sci. U.S.A.
  \textbf{36} (1950), no.~1, 192--197.

\bibitem{firey1965lower}
William~J. Firey, \emph{Lower bounds for volumes of convex bodies}, Arch. Math.
  \textbf{16} (1965), 69--74.

\bibitem{fodornaszoditamas2018on}
Ferenc Fodor, M\'{a}rton Nasz\'{o}di, and Tam\'{a}s Zarn\'{o}cz, \emph{On the
  volume bound in the {D}voretzky--{R}ogers lemma}, Pacific J. Math. (2019), to
  appear.

\bibitem{gritzklee1992inner}
Peter Gritzmann and Viktor Klee, \emph{Inner and outer j-radii of convex bodies
  in finite-dimensional normed spaces}, Discrete Comput. Geom. \textbf{7}
  (1992), no.~1, 255--280.

\bibitem{groemer1966zusammen}
Helmut Groemer, \emph{Zusammenh\"angende {L}agerungen konvexer {K}\"orper},
  Math. Z. \textbf{94} (1966), 66--78.

\bibitem{gruber2007convex}
Peter~M. Gruber, \emph{Convex and {D}iscrete {G}eometry}, Grundlehren der
  Mathematischen Wissenschaften, vol. 336, Springer-Verlag, Berlin, 2007.

\bibitem{hardylittlewoodpolya1952inequalities}
Godfrey~H. Hardy, John~E. Littlewood, and George P\'{o}lya,
  \emph{Inequalities}, Cambridge, at the University Press, 1952, 2nd ed.

\bibitem{heil1978kleinste}
Erhard Heil, \emph{Kleinste konvexe {K}{\"o}rper gegebener {D}icke}, Preprint
  453, TU Darmstadt (1978).

\bibitem{henk2012loewnerjohn}
Martin Henk, \emph{L\"owner-{J}ohn ellipsoids}, Doc. Math. (2012), 95--106,
  Extra vol.: Optimization stories.

\bibitem{john2014extremum}
Fritz John, \emph{Extremum problems with inequalities as subsidiary
  conditions}, Studies and {E}ssays {P}resented to {R}. {C}ourant on his 60th
  {B}irthday, {J}anuary 8, 1948, Interscience Publishers, Inc., New York, N.
  Y., 1948, pp.~187--204.

\bibitem{lemmens1973equiangular}
Piet W.~H. Lemmens and Johan~Jacob Seidel, \emph{Equiangular lines}, J. Algebra
  \textbf{24} (1973), 494--512.

\bibitem{MaherFischerGallyetal.2017}
Stephen~J. Maher, Tobias Fischer, Tristan Gally, Gerald Gamrath, Ambros
  Gleixner, Robert~Lion Gottwald, Gregor Hendel, Thorsten Koch, Marco~E.
  L{\"u}bbecke, Matthias Miltenberger, Benjamin M{\"u}ller, Marc~E. Pfetsch,
  Christian Puchert, Daniel Rehfeldt, Sebastian Schenker, Robert Schwarz,
  Felipe Serrano, Yuji Shinano, Dieter Weninger, Jonas~T. Witt, and Jakob
  Witzig, \emph{The {SCIP} {O}ptimization {S}uite 4.0}, Tech. Report 17-12,
  ZIB, Takustr.~7, 14195 Berlin, 2017.

\bibitem{makai1978on}
Endre Makai~Jr., \emph{On the thinnest non-separable lattice of convex bodies},
  Studia Sci. Math. Hungar. \textbf{13} (1978), 19--27.

\bibitem{pal1921ein}
Julius P\'al, \emph{Ein {M}inimumproblem f\"ur {O}vale}, Math. Ann. \textbf{83}
  (1921), no.~3-4, 311--319.

\bibitem{pelczynski1991parallelepipeds}
Aleksander Pe\l{c}zy\'{n}ski and Stanis\l{a}w~J. Szarek, \emph{On
  parallelepipeds of minimal volume containing a convex symmetric body in
  {${\bf R}^n$}}, Math. Proc. Cambridge Philos. Soc. \textbf{109} (1991),
  no.~1, 125--148.

\bibitem{tomczak1989banach}
Nicole Tomczak-Jaegermann, \emph{Banach-{M}azur distances and
  finite-dimensional operator ideals}, Pitman Monographs and Surveys in Pure
  and Applied Mathematics, vol.~38, Longman Scientific \& Technical, Harlow;
  copublished in the United States with John Wiley \& Sons, Inc., New York,
  1989.

\end{thebibliography}

\end{document}